\documentclass[11pt]{amsart}
\usepackage{amsfonts, amssymb, amsmath, mathrsfs}
\usepackage{tikz-cd}

\usepackage[hidelinks]{hyperref}
\usepackage{url}

\usepackage[utf8]{inputenc}

\usepackage{bbm}

\newcommand{\F}{\mathbb{F}}
\newcommand{\Fp}{\mathbb{F}_p}

\newcommand{\zli}{\zeta_{\ell^{\infty}}}

\newcommand{\KK}{\mathbb{K}}
\newcommand{\Q}{\mathbb{Q}}
\newcommand{\Qab}{\mathbb{Q}^{\mathrm{ab}}}
\newcommand{\Z}{\mathbb{Z}}

\newcommand{\hZ}{\widehat{\Z}}

\newcommand{\Zl}{\mathbb{Z}_{\ell}}
\newcommand{\Qp}{\mathbb{Q}_{p}}
\newcommand{\Qpbar}{\overline{\mathbb{Q}}_{p}}
\newcommand{\Ql}{\mathbb{Q}_{\ell}}

\newcommand{\kbar}{\overline{k}}
\newcommand{\Qlbar}{\overline{\mathbb{Q}}_{\ell}}

\newcommand{\Qbar}{\overline{\mathbb{Q}}}

\newcommand{\calF}{\mathcal{F}}

\newcommand{\Rep}{\mathbf{Rep}}

\DeclareMathOperator{\Spec}{Spec}
\DeclareMathOperator{\Gal}{Gal}

\DeclareMathOperator{\Aut}{Aut}

\DeclareMathOperator{\GL}{GL}

\newcommand{\Pibar}{\overline{\Pi}}

\newcommand{\et}{\textrm{ét}}

\begin{document}

\newtheorem{theo}[subsection]{Theorem}
\newtheorem*{theo*}{Theorem}
\newtheorem{conj}{Conjecture}

\renewcommand{\theconj}{\Alph{conj}}
\newtheorem*{conj*}{Conjecture}
\newtheorem{prop}[subsection]{Proposition}
\newtheorem{lemm}[subsection]{Lemma}
\newtheorem*{lemm*}{Lemma}
\newtheorem{coro}[subsection]{Corollary}

\theoremstyle{definition}
\newtheorem{defi}[subsection]{Definition}
\newtheorem*{defi*}{Definition}
\newtheorem{cons}[subsection]{Construction}
\newtheorem{ques}[subsection]{Question}
\newtheorem{rema}[subsection]{Remark}
\newtheorem{exam}[subsection]{Example}

\numberwithin{equation}{section}

\title{On the semi-simplicity conjecture for $\Qab$}
\author{Marco D'Addezio}

\address{Max-Planck-Institut für Mathematik, Vivatsgasse 7, 53111, Bonn, Germany}
\email{ daddezio@mpim-bonn.mpg.de}
\begin{abstract}
	We show that the semi-simplicity conjecture for finitely generated fields follows from the conjunction of the semi-simplicity conjecture for finite fields and for the maximal abelian extension of the field of rational numbers.
\end{abstract}
\maketitle

\section{Notation}

\subsection{}
 Let $k$ be a field, $\overline{k}$ an algebraic closure and $\ell$ a prime number different from the characteristic of $k$. Write $\Gamma_k$ for the Galois group $\Gal(\overline{k}/k)$. We denote by $\Rep_{\Ql}(\Gamma_{k})$ the neutral Tannakian category of continuous finite-dimensional $\Ql$-linear representations of $\Gamma_k$. We shall refer to the objects in $\Rep_{\Ql}(\Gamma_{k})$ simply as \textit{($\ell$-adic) representations of $\Gamma_k$}.
 
Let $\Rep^{\mathrm{geo}}_{\Ql}(\Gamma_{k})$ be the smallest neutral Tannakian subcategory of $\Rep_{\Ql}(\Gamma_{k})$, closed under subquotients, which contains all the $\ell$-adic representations of $\Gamma_k$ of the form $H^i_{\et}(X_{\kbar},\Ql)$, where $i$ is an integer, $X$ is a smooth and projective variety over $k$ and $X_{\kbar}:=X\otimes_k \kbar$. We shall say that an object in $\Rep^{\mathrm{geo}}_{\Ql}(\Gamma_{k})$ is an \textit{$\ell$-adic representation coming from geometry}.
\subsection{}
Let $\KK$ be a field of characteristic $0$ and $V$ a finite-dimensional $\KK$-vector space. We say that a linear endomorphism $\varphi$ of $V$ is \textit{semi-simple} if it is diagonalizable after a finite extension of $\KK$. 
Let $V_\rho$ be an $\ell$-adic representation of $\Gamma_\Q$ and $p\neq\ell$ a prime number where $V_\rho$ is unramified, we say that $\rho$ is \textit{semi-simple at $p$} if one (or equivalently any) Frobenius element at $p$ acts via a semi-simple automorphism.

\subsection{}
For a ring $R$ and a positive integer $n$ we write $R[\zeta_n]$ for the ring quotient $R[t]/(t^n-1)$ and $R[\zeta_\infty]$ for the ring colimit $\varinjlim_{n} R[\zeta_n]$. If $\ell$ is a prime number, we denote by $R[\zeta_{\ell^{\infty}}]$ the ring $\varinjlim_{n} R[\zeta_{\ell^n}]$. Besides, we denote by $\Qab$ the maximal abelian extension of $\Q$ in $\Qbar$.
\section{Introduction}
\subsection{}Let $k$ be a field, we consider the following statement.
\begin{itemize}
	\item[] $S(k)$: For every prime number $\ell$ different from the characteristic of $k$, an $\ell$-adic representation of $\Gamma_k$ coming from geometry is semi-simple.
\end{itemize}
\noindent
Grothendieck and Serre conjectured that for every finitely generated field $k$, the assertion $S(k)$ is true, \cite{Tat}. This conjecture is commonly known as the \textit{semi-simplicity conjecture}. Note that the conjecture predicts that $S(k)$ is true even for fields that are infinite Galois extensions of a finitely generated field. Indeed, if $k'/k$ is a Galois extension then $S(k)$ implies $S(k')$ because the restriction of a semi-simple representation to a normal subgroup is semi-simple. For this reason, Grothendieck--Serre semi-simplicity conjecture predicts, for example, that the $\ell$-adic representions of $\Gamma_{\Qab}$ coming from geometry are semi-simple. On the other hand, it is worth recalling that $S(k)$ is false in general. For example, over the local fields $\Qp$ and $\mathbb{C}((t))$ the representations coming from geometry are not semi-simple in general. 

In this article, we prove the following implication.

\begin{theo}[Theorem \ref{main:t}]
Let $k$ be a Galois extension of a finitely generated field. The conjunction of $S(\Fp)$ for every prime number $p$ and $S(\Qab)$ implies $S(k)$.
\end{theo}

\subsection{}Let us make a brief summary on what is already known about Grothendieck--Serre semi-simplicity conjecture. The first result was obtained in 1948 by Weil, who proved the conjecture for abelian varieties (and hence for curves) over finite fields. In this case, the semi-simplicity follows from the positivity of the Rosati involution. Later, in 1983, Faltings proved the semi-simplicity conjecture for abelian varieties over number fields, as an intermediate step of his proof of the Mordell conjecture. By the work of Deligne, both these results extend to K3 surfaces, thanks to the Kuga--Satake construction.

In 1980, Deligne obtained a general semi-simplicity result in positive characteristic, as a consequence of his theory of weights.

\begin{theo}[{\normalfont\cite[Théorème 3.4.1.(iii)]{Weil2}}]
	\label{deligne:t}
	Let $X$ be a normal scheme of finite type over $\Fp$. For every $\iota$-pure lisse $\Qlbar$-sheaf over $X$, the inverse image over $X_{\overline{\F}_p}$ is semi-simple. In particular, for every finitely generated field extension $k_{\infty}/\overline{\F}_p$, the assertion $S(k_{\infty})$ is true.
\end{theo}
Note that Theorem \ref{deligne:t} is related to the semi-simplicity conjecture because when $k$ is a finitely generated field extension of $\Fp$, the representations of $k$ coming from geometry are direct sum of pure representations, \cite{Weil1}. Using Theorem \ref{deligne:t}, Lei Fu proved the following.

\begin{theo}[\cite{Fu}]\label{Fu:t}
Let $X$ be a normal connected scheme of finite type over
$\F_p$ and let $\calF$ be a $\iota$-pure lisse $\Qlbar$-sheaf over $X$. If there exists a closed point $x$ of $X$ such that the Frobenius automorphism of $\calF$ at $x$ is semi-simple, then $\calF$ is a
semi-simple lisse sheaf over $X$. In particular, for every finitely generated field $k$ of positive characteristic $p$, the assertion $S(\Fp)$ implies $S(k)$.
\end{theo}
The idea of the proof of Theorem \ref{Fu:t} is to use the exact sequence 
\begin{equation}\label{k:eq}
1\to \Gamma_{k\overline{\F}_p}\to \Gamma_{k}\to\Gal(\overline{\F}_p/\F_p)\to 1 
\end{equation}
where $k$ is the function field of $X$. Thanks to (\ref{Fu:t}) one can combine the semi-simplicity of the restriction to $\Gamma_{k\overline{\F}_p}$ provided by Theorem \ref{deligne:t} and the condition at $x$ in order to get the semi-simplicity of the entire representation of $\Gamma_{k}$.

We end this section mentioning the main general result which is known so far on the semi-simplicity conjecture in characteristic $0$.
\begin{theo}[\cite{Serre}]
	\label{s-s-char-0:t} If $k$ is a finitely generated field of characteristic $0$, then $S(\Q)$ implies $S(k)$.
\end{theo}
The theorem is an application of Serre's specialization method via Hilbert's irreducibility theorem.

\section{Some analogies}\label{anal:s}

Our work is based on the analogy between the field of rational numbers and the function field $\F_p(t)$, or more precisely between the Galois extensions $\Qab/\Q$ and $\overline{\F}_p(t)/{\F}_p(t)$. Note that both extensions are the maximal cyclotomic extensions of the respective base fields. The Galois group $\Gal(\Qab/\Q)$ is canonically isomorphic to $\hZ^{\times}$. If we denote by $\delta:\Gamma_{\Q}\twoheadrightarrow \hZ^{\times}$ the quotient induced by this identification, we get an exact sequence 
\begin{equation}\label{Q:eq}
1\to \Gamma_{\Qab}\to \Gamma_{\Q}\xrightarrow{\delta} \hZ^{\times}\to 1,
\end{equation}which is the analogue of (\ref{k:eq}). Following this analogy, we intend to investigate in this article the following conjecture which is inspired by Theorem \ref{Fu:t}.

\begin{conj}
\label{conj-A:c}
If an $\ell$-adic representation of $\Gamma_\Q$ coming from geometry is semi-simple at some unramified prime number $p$ different from $\ell$, then it is semi-simple as a representation of $\Gamma_\Q$.
\end{conj}

We shall show in the next section how to adapt Fu's proof of Theorem \ref{Fu:t} in order to prove that $S(\Qab)$ implies Conjecture \ref{conj-A:c}. But before going into further details, we would like to speculate a bit more on $S(\Qab)$. Continuing the previous analogy, we wonder whether is it possible to prove $S(\Qab)$, as for Theorem \ref{deligne:t}, via a suitable theory of weights for the $\ell$-adic representations of $\Gamma_{\Qab}$. Note that one cannot hope that every pure $\ell$-adic representation of $\Gamma_{\Q}$ is semi-simple when restricted to $\Gamma_{\Qab}$, as we illustrate in the following example.

\begin{exam}
	\label{exte:ex}
Let $K/\Q$ be an imaginary quadratic extension and let $K_{\infty}^-/K$ be the \textit{anti-cyclotomic $\Zl$-extension} of $K$, namely that $\Zl$-extension of $K$ which is also a non-abelian Galois extension of $\Q$. We choose an isomorphism of the Galois group $\Gal(K_{\infty}^-/\Q)$ with $\Zl\rtimes \Z/2$, where $\Zl$ corresponds to $\Gal(K_{\infty}^-/K)$ and $\Z/2$ acts non-trivially on $\Zl$. Write $\chi$ for the non-trivial character of $\Gal(K/\Q)$. We claim that there exists a non-trivial extension of $\ell$-adic representations of $\Gamma_{\Q}$ $$0\to \chi \to V\to \Ql\to 0$$ which becomes trivial when restricted to $\Gamma_{K_{\infty}^-}$. This is constructed by mapping $$(1,0)\mapsto \begin{pmatrix}
1 & 1 \\
0 & 1
\end{pmatrix} \textrm{ and  } (0,1)\mapsto \begin{pmatrix}
-1 & 0 \\
0 & 1
\end{pmatrix}.$$ This extension is non-trivial when restricted to $\Gamma_{\Qab}$ because $K_{\infty}^-$ is not in $\Qab$. On the other hand, $V$ is manifestly pure of weight $0$.
\end{exam}

In order to exclude extensions as the one presented in the previous example, it is reasonable to add conditions on the $\ell$-adic representations considered using \textit{$\ell$-adic Hodge theory}. For example, in Example \ref{exte:ex}, since $V|_{\Gamma_{\Ql^{\mathrm{ab}}}}$ is a unipotent non-trivial representation, it is not \textit{Hodge--Tate}, \cite[§2.4.5]{BC09}. On the other hand, the $\ell$-adic representations coming from geometry are \textit{de Rham} at $\ell$, \cite{Fal89}. Let us explain a refined hope.

\subsection{}

In order to prove $S(\Qab)$, one has to show that extensions of representations of $\Gamma_{\Qab}$ coming from geometry are all trivial. It is easy to see that, without loss of generality, one can simply consider extensions of the trivial representation by another representation.

Let $V$ be an $\ell$-adic representation of $\Gamma_\Q$ coming from geometry and of weight $0$ and let $N$ be a multiple of the product of all the prime numbers where $V$ is ramified. We consider the vector space $$H:=H^1_{\et}(\Spec(\Z[\zeta_\infty,N^{-1}]),V)$$endowed with the left action of the group $\Aut(\Z[\zeta_\infty,N^{-1}])=\hZ^\times$ acting by pushforward. The group $H$ parametrises all the extensions of $\Ql$ by $V|_{\Gamma_{\Qab}}$ which are unramified away of $N$. We choose a prime number $p\nmid N\ell$ and a lift of $p$ via the quotient map $\hZ^\times\twoheadrightarrow (\hZ/\Z_p)^{\times}$, denoted by $\widetilde{p}\in \hZ^\times$. Write $\varphi_{\tilde{p}}$ for the endomorphism of $H$ induced by $\tilde{p}$. Also, let $H_{g} \subseteq H$ be the \textit{Selmer group} obtained by imposing local de Rham conditions at $\ell$, as in \cite{BK90}. Suppose that the following assumption is true.
	
	\begin{itemize}
		\item[$W(V,N)$:] For every eigenvalue $\alpha$ of $\varphi_{\tilde{p}}$ acting on $H_{g}$ there exists an embedding $\iota: \Qlbar\hookrightarrow \mathbb{C}$ such that $|\iota(\alpha)|>1$.
	\end{itemize}
	 Then it follows that $\varphi_{\tilde{p}}$ acts without fixed points on $H_{g}$. In turn, this implies that every extension of $\Ql$ by $V$ over $\Spec(\Z[\zeta_\infty,N^{-1}])$ which descends to $\Spec(\Z[N^{-1}])$ and comes from geometry is trivial. We proved the following.
	 
	 \begin{prop}
	 	
If $W(V,N)$ is true for every $V$ and $N$ as above, then $S(\Qab)$ is true.
	 \end{prop}

\section{Our main results}
We choose a closed embedding $\Gamma_{\Qp}\subseteq \Gamma_\Q$ induced by a field embedding $\Qbar\hookrightarrow \Qpbar$ and a Frobenius lift $F_p\in \Gamma_{\Qp}\subseteq \Gamma_\Q$. We want to prove in this section the following result.
\begin{theo}
	\label{pt-ab:t}Let $\rho$ be an $\ell$-adic representation of $\Gamma_\Q$ which is semi-simple when restricted to $\Gamma_{\Qab}$. If there exists a prime number $p\neq \ell$ and a Frobenius element $F_p\in \Gamma_\Q$ such that $\rho(F_p)$ is semi-simple, then $\rho$ is a semi-simple representation of $\Gamma_\Q$. In particular, $S(\Qab)$ implies Conjecture \ref{conj-A:c}.
\end{theo}
Before going into the proof, let us first see how to deduce from Theorem \ref{pt-ab:t} the main result of our article.
\begin{theo}\label{main:t}
	Let $k$ be a Galois extension of a finitely generated field. The conjunction of $S(\Fp)$ for every prime number $p$ and $S(\Qab)$ implies $S(k)$.
\end{theo}

\begin{proof}
If $k$ has positive characteristic $p$, thanks to Theorem \ref{Fu:t}, we have that $S(\Fp)$ implies $S(k)$. We pass to characteristic $0$. Thanks to Theorem \ref{s-s-char-0:t}, it is enough to prove $S(\Q)$. Also, in light of Theorem \ref{pt-ab:t}, we know that $S(\Qab)$ implies Conjecture \ref{conj-A:c}. Let us show how to deduce $S(\Q)$ from here. Let $X$ be a smooth projective variety over $\Q$ and let $i$ be a natural number. We choose a prime number $p$ where $X$ admits good reduction $\overline{X}/\Fp$. Thanks to $S(\Fp)$, we know that the Frobenius acting on $H^i_{\et}(\overline{X}_{\overline{\mathbb{F}}_p},\Ql)$ is semi-simple. By the smooth and proper base-change theorem, the action of $F_p$ on $H^i_{\et}(X_{\overline{\mathbb{Q}}},\Ql)$ is semi-simple as well. By virtue of Conjecture \ref{conj-A:c}, this implies that the representation of $\Gamma_{\Q}$ on $H^i_{\et}(X_{\overline{\mathbb{Q}}},\Ql)$ is semi-simple. Since this holds for every $X$ and $i$, we get $S(\Q)$. 	
\end{proof}
In order to prove Theorem \ref{pt-ab:t}, we adapt Fu's proof in \cite{Fu} to our situation following the analogy in §\ref{anal:s}. For this purpose, we introduce an \textit{ad hoc} notion of a Weil group of $\Q$. 
\begin{defi}For every $n\in \Z$, the element $F_p^n\in \Gamma_{\Q}$ acts on $\Gamma_{\Qab}$ by conjugation. This induces a continuous action of $\Z$, with the discrete topology, on $\Gamma_{\Qab}$. Let $W_{\Q,F_p}$ be semi-direct product $\Gamma_{\Qab}\rtimes \Z$ as topological groups, where $\Z$ acts as above. We say that $W_{\Q,F_p}$ is the \textit{Weil group of $\Q$ with respect to $F_p$}.
\end{defi}
\subsection{}

The group $W_{\Q,F_p}$ sit in the following commutative diagram with exact rows.

\begin{equation}
\begin{tikzcd}
\label{w:eq}
1 \arrow{r}& \Gamma_{\Qab} \arrow{r}\arrow[d,equal] &W_{\Q,F_p} \arrow{r}\arrow[d,hook]& \Z \arrow{r}\arrow[d,hook]&1\\
1\arrow{r} & \Gamma_{\Qab} \arrow{r} & \Gamma_{\Q}  \arrow{r} & \hZ^{\times}\arrow{r} &1
\end{tikzcd}
\end{equation}
where the injective map $W_{\Q,F_p}\hookrightarrow \Gamma_\Q$ sends $1\in \Z$ to $F_p$. Following the proof of \cite{Fu}, it is rather straightforward to prove that if $\rho$ is a representation satisfying the hypothesis in Theorem \ref{pt-ab:t}, then $\rho|_{W_{\Q,F_p}}$ is semi-simple. The additional issue in our set-up is that the subgroup $W_{\Q,F_p}$ is not dense in $\Gamma_{\Q}$. The closure of $W_{\Q,F_p}$ in $\Gamma_{\Q}$ is not even of finite index. Nonetheless, it is still possible to prove the following result.

\begin{prop}\label{rest-to-Weil:p}
	Let $\rho$ be an $\ell$-adic representation of $\Gamma_{\Q}$. If for some prime number $p\neq\ell$ and for some choice of $F_p$ its restriction to $W_{\Q,F_p}$ is semi-simple, then $\rho$ is a semi-simple representation of $\Gamma_{\Q}$.
\end{prop}

In order to prove Proposition \ref{rest-to-Weil:p}, we need a couple of preparatory lemmas. We start by recalling a well-known fact on the quotients of $\hZ^\times$. It is worth mentioning that this lemma is also one of the main ingredients in the proof of Moonen's recent result on the semi-simplicity conjecture, \cite{Moo}.
\begin{lemm}\label{Iwas:l}
	There exits a unique continuous quotient $\hZ^{\times}\twoheadrightarrow \Zl$ up to automorphisms of the target.
\end{lemm}
\begin{proof}
By the Chinese remainder theorem, $\hZ^{\times}$ is isomorphic to the product $\prod_{\ell'}\Z_{\ell'}^\times$ where $\ell'$ varies among all the prime numbers. We choose a continuous isomorphism between $\Zl^{\times}/torsion$ and $\Zl$ which, in turn, induces a quotient $\pi:\hZ^{\times}\twoheadrightarrow \Zl^{\times}/torsion=\Zl$. 

Passing to the unicity, let $\pi':\hZ^{\times}\twoheadrightarrow\Zl$ be another $\Zl$-quotient of $\hZ^{\times}$. We want to check that $\pi$ and $\pi'$ are the same up to automorphisms of $\Zl$. To do this, we may replace $\hZ^{\times}$ with the dense subgroup $\bigoplus_{\ell'} \Z_{\ell'}^{\times}\subseteq \prod_{\ell'}\Z_{\ell'}^\times$. Since $\Zl$ is torsion-free and there are no non-trivial morphisms $\Z_{\ell'}\to\Zl$ when $\ell'\neq \ell$, every continuous morphism from $\bigoplus_{\ell'} \Z_{\ell'}^{\times}$ to $\Zl$ factors through $\pi$. Therefore, there exists a continuous endomorphism $\alpha$ of $\Zl$ such that $\pi'=\alpha\circ \pi$. Since $\pi'$ is surjective, $\alpha$ is surjective. This implies that $\alpha$ is an automorphism, as we wanted.
\end{proof}
 We fix from now on a quotient morphism $\delta_\ell: \Gamma_{\Q}\twoheadrightarrow \Zl$.

\begin{cons}
Let $\rho$ be an $\ell$-adic representation of $\Gamma_\Q$. We write $\Pi$ for the image of $\Gamma_\Q$ and $\Pi^0$ for the image of $\Gamma_{\Qab}$. We also denote by $\Pibar$ the quotient $\Pi/\Pi^0$ and by $\pi$ the natural projection $\pi:\Pi\twoheadrightarrow \Pibar$. We obtain the following commutative diagram of profinite groups with exact rows

\begin{center}
	\begin{tikzcd}
	1 \ar{r}& \Gamma_{\Qab} \ar{r}\ar[d,two heads,"\rho"] &\Gamma_\Q \ar[r,"\delta"]\ar[d,two heads,"\rho"]& \hZ^{\times} \ar[r]\ar[d,two heads]& 1\\
	1\ar{r} & \Pi^0 \ar{r} & \Pi  \ar[r,"\pi"] & \Pibar\ar{r} & 1.
	\end{tikzcd}
\end{center}
\end{cons}

\begin{lemm}
	\label{pibar:l}
	The group $\Pibar$ is either a finite abelian group or it admits a surjective morphism $\overline{\pi}_\ell: \Pibar \twoheadrightarrow \Zl$ with finite kernel such that $\delta_\ell=\overline{\pi}_\ell \circ \pi \circ \rho$. 
\end{lemm}
\begin{proof}

The group $\Pi$ is a closed subgroup of the topological group $\GL(V_\rho)$, thus it can be endowed with the structure of an \textit{$\ell$-adic analytic group}, \cite[Theorem 9.6]{DSMS}. By [\textit{ibid.}, Theorem 8.32], this implies that $\Pi$ contains an open topologically finitely generated pro-$\ell$-subgroup. Since $\pi$ is surjective, the same is then true for $\Pibar$. Besides, note that the group $\Pibar$, being a quotient of $\hZ^\times$, is an abelian group. 

Let $\Pibar_\ell\subseteq \Pibar$ be the maximal pro-$\ell$-subgroup of $\Pibar$. By the previous, $\Pibar_\ell$ is a finitely generated $\Zl$-module. Moreover, since $\Pibar$ is abelian, the inclusion $\Pibar_\ell\subseteq \Pibar$ admits a retraction. This implies that
 there exists a surjective morphism $\hZ^{\times}\twoheadrightarrow\Pibar_\ell$. By Lemma \ref{Iwas:l}, $\Pibar_\ell$ is a $\Zl$-module of rank at most 1. This means precisely that either $\Pibar$ is finite or it admits a quotient $\overline{\pi}_\ell:\Pibar\twoheadrightarrow\Zl$ with finite kernel. In the second case, thanks to Lemma \ref{Iwas:l}, up to post-composing $\overline{\pi}_\ell$ with an automorphism of $\Zl$, the quotients $\delta_\ell$ and $\overline{\pi}_\ell \circ \pi \circ \rho$ are equal. This concludes the proof.

\end{proof}

\subsection{} \emph{Proof of Proposition \ref{rest-to-Weil:p}.}
Let $\Pi_1$ be the closure of the image of $W_{\Q,F_p}$ in $\Pi$. In light of \cite[Lemma 1]{Fu}, it is enough to show that $\Pi_1$ has finite index in $\Pi$. We note that since $\Pi_1$ contains $\Pi^0$, if we set $\Pibar_1:=\pi(\Pi_1)$, then $[\Pi: \Pi_1]=[\Pibar: \Pibar_1]$. Thanks to this we are reduced to showing that $\Pibar_1$ has finite index in $\Pibar$. 

By Lemma \ref{pibar:l}, we have two cases. If $\Pibar$ is finite the result holds trivially. If $\Pibar$ is infinite, we take $\overline{\pi}_\ell: \Pibar\twoheadrightarrow\Zl$ as in the lemma. Since $\overline{\pi}_\ell$ has finite kernel, to prove that $\Pibar_1$ has finite index in $\Pibar$ it is enough to show that $\overline{\pi}_\ell(\Pibar_1)$ has finite index in $\Zl$. The profinite group $\Pibar_1$ is topologically generated by $\rho(F_p)$, therefore the group $\overline{\pi}_\ell(\Pibar_1)$ is topologically generated by $\delta_\ell(F_p)$. We conclude the proof by virtue of the following lemma.

\qed
\begin{lemm}
	\label{image-p:l}
	If $\ell\neq p$, the closure of the group generated by $\delta_\ell(F_p)$ in $\Zl$ is an open subgroup.
\end{lemm}
\begin{proof}
	Let $H$ be the closure of the group generated by $\delta_\ell(F_p)$ in $\Zl$. Since $H$ is a subgroup, for every $n\in \Z$ the multiplication by $n$ in $\Zl$ maps $H$ to $H$. This implies that $H$ is an ideal of the DVR $\Zl$. It remains to show that $H\neq (0)$ or, equivalently, that $\delta_\ell(F_p)\neq 0$.
	
	By Lemma \ref{Iwas:l}, the quotient $\delta_\ell$ factors through $\Gal(\Q(\zli)/\Q)$. Let $\Qp^{\mathrm{ur}}$ be the maximal unramified extension of $\Qp$ in $\Qpbar$. The restriction of $\delta_\ell$ to $\Gamma_{\Qp}$ factors through $\Gal(\Qp^{\mathrm{ur}}/\Qp)$, because the extension $\Q(\zli)/\Q$ is unramified at $p$. To summarize we have the following commutative diagram
	
		\begin{center}
		\begin{tikzcd}
		\Gamma_{\Qp}\ar[r,two heads]\ar[d,hook] &\Gal(\Qp^{\mathrm{ur}}/\Qp) \ar[rd, bend left=25,two heads]\ar[d]\\
	    \Gamma_{\Q} \ar[rr, bend right=15, "\delta_\ell"',two heads] \ar[r,two heads] & \Gal(\Q(\zli)/\Q)\ar[r,two heads] & \Zl.
		\end{tikzcd}
	\end{center}

	By definition, the image of $F_{p}\in \Gamma_{\Qp}$ in $\Gal(\Qp^{\mathrm{ur}}/\Qp)$ is the Frobenius lift. Therefore, the image of $F_p$ in $\Gal(\Q(\zli)/\Q)$ is the unique automorphism of $\Q(\zli)$ which raises each root of unit to its $p$-th power. This automorphism has manifestly infinite order in $\Gal(\Q(\zli)/\Q)$. Since the quotient $\Gal(\Q(\zli)/\Q)\twoheadrightarrow \Zl$ has finite kernel, this implies that $\delta_\ell(F_p)\neq 0$, as we wanted.

\end{proof}

\subsection{}\emph{Proof of Theorem \ref{pt-ab:t}}. By Proposition \ref{rest-to-Weil:p}, it is enough to check that $\rho$ is semi-simple when restricted to $W_{\Q,F_p}\subseteq\Gamma_{\Q}$. Let $G^0$ be the Zariski closure of $\Pi^0$ in $\GL(V_\rho)$ and let $G_{F_p}$ be the semi-direct product $G^0\rtimes \Z$ as group schemes where $1\in \Z$ acts on $G^0$ as $\rho(F_p)$ acts on $G^0$ by conjugation. We define $\widetilde{\rho}:W_{\Q,F_p}\to G_{F_p}$ as the only morphism making the following diagram commuting
\begin{center}
	\begin{tikzcd}
		1 \ar{r}& \Gamma_{\Qab} \ar{r}\ar[d,"\rho|_{\Gamma_{\Qab}}"] &W_{\Q,F_p} \ar[r]\ar[d,"\widetilde{\rho}"]& \Z \ar[r]\ar[d,equal]& 1\\
		1\ar{r} & G^0 \ar{r} & G_{F_p}\ar[r] & \Z \ar{r}& 1.
	\end{tikzcd}
\end{center}
Let $\sigma:G_{F_p}\to \GL(V_\rho)$ be the representation which extends the tautological representation $G^0\hookrightarrow \GL(V_\rho)$ by sending $1\in \Z$ to $\rho(F_p)$. The composition $\sigma \circ \widetilde{\rho}: W_{\Q,F_p}\to \GL(V_\rho)$ is equal to $\rho$. Since $\widetilde{\rho}$ has Zariski-dense image, in order to show that the restriction of $\rho$ to $W_{\Q,F_p}$ is semi-simple it is enough to check that $\sigma$ is semi-simple.

The group $G^0$ is reductive because, by assumption, $\rho$ is semi-simple when restricted to $\Gamma_{\Qab}$. Thanks to \cite[Lemme 1.3.10]{Weil2}, there exists a $\Qlbar$-point $g$ in the centre of $G_{F_p}$ of the form $(g',d)$ where $g'\in G^0(\Qlbar)$ and $d\neq 0$. Since $\rho(F_p)$ is semi-simple, we may apply \cite[Lemma 2]{Fu} to the representation $\sigma$, obtaining thereby the desired result.

\qed

\section*{Acknowledgments}

I thank my advisor Hélène Esnault and Matteo Tamiozzo for all the valuable remarks on a first draft of this article. I also thank Ben Moonen, Kazuya Kato, Daniel Litt and Peter Scholze for enlightening discussions on the topic. Finally, I thank the anonymous referees for helpful suggestions.

\bibliographystyle{ams-alpha}

\end{document}